\documentclass[12pt]{article}

\usepackage{stmaryrd} \newcommand{\rbose}{\rrbracket}\newcommand{\lbose}{\llbracket}

\usepackage{pifont}
\usepackage{amssymb}
\usepackage{amsmath}
\usepackage{fancyhdr}
\usepackage{setspace}
\usepackage{enumitem}
\setlist{nolistsep}
\usepackage{graphicx}

\usepackage[usenames,dvipsnames]{color}
\usepackage{epsfig}

\usepackage[mathscr]{eucal}

\newtheorem{theorem}{Theorem}[section]
\newtheorem{lemma}[theorem]{Lemma}
\newtheorem{corollary}[theorem]{Corollary}
\newtheorem{defn}[theorem]{Definition}

\newtheorem{remark}[theorem]{Remark}

\newtheorem{result}[theorem]{Result}
\newenvironment{proof}{\noindent{\bf Proof}\hspace{0.5em}}
    { \null  \hfill $\square$ \par}

\newcommand{\scroll}{\mathscr S}

    \newcommand{\Fq}{\mathbb F_{\!q\!\;}}
\newcommand{\Fqqq}{\mathbb F_{q^3}}
\newcommand{\Fqq}{\mathbb F_{\!q^2\,}}

\newcommand{\plus}{{\raisebox{.2\height}{\scalebox{.5}{+}}}}
\newcommand{\Cplus}{\C^{{\rm \pmb\plus}}}
\newcommand{\si}{\Sigma_\infty}

\newcommand\B{{\pi}}
\newcommand\bs{{b}}
\newcommand\C{{\mathcal  C}}
\newcommand{\Q}{\mathscr  Q}
\renewcommand{\O}{\mathscr O}
\newcommand{\Pit}{\Gamma}

\newcommand\U{{\cal U}}

\newcommand{\R}{\mathcal R}

\newcommand\V{{\cal V}}
\newcommand\I{{\cal I}}
\newcommand\W{{\cal W}}

\newcommand\F{{\cal F}}
\newcommand\N{{\cal N}}

\renewcommand{\S}{\mathcal S}

\newcommand{\K}{\mathcal K}

\newcommand{\VO}{\V(\Bo{\O})}
\newcommand{\VC}{\V(\Bo{\C})}
\newcommand{\VCplus}{\V(\Bo{\Cplus})}
\newcommand{\Vb}{\V(\Bo{\bs})}
\newcommand{\VB}{\V(\Bo{\B})}

\newcommand{\li}{\ell_\infty}
\newcommand{\st}{\,|\,}

\newcommand\PGammaL{{\mbox{P}\Gamma {\mbox L}}}
\newcommand\PGL{{\rm PGL}}

\newcommand\PG{{\rm PG}}

\renewcommand{\star}{{^{\mbox{\tiny\ding{73}}}}}
\newcommand\Bo[1]{\mbox{$\lbose#1\rbose$}}

\newcommand{\gstar}{g^{\mbox{\tiny\ding{72}}}}
\newcommand{\gqstar}{{g^q}^{\mbox{\tiny\ding{72}}}}

\newcommand{\Pirstar}{{\Pi}^{\mbox{\tiny\ding{73}}}_r}
\newcommand{\Pirstarstar}{{\Pi}^{\mbox{\tiny\ding{72}}}_r}
\newcommand{\Pigstar}{{\Pi}^{\mbox{\tiny\ding{73}}}_g}
\newcommand{\PiUstar}{{\Pi}^{\mbox{\tiny\ding{73}}}_{\U}}

\newcommand{\Sigmabstar}{{\Sigma}^{\mbox{\tiny\ding{73}}}_b}

\newcommand{\Sonetwostar}{{\S}^{\mbox{\tiny\ding{73}}}_{1;2}}

\newcommand{\newQonestar}{{\Q}^{\mbox{\tiny\ding{73}}}_1}
\newcommand{\newQthreestar}{{\Q}^{\mbox{\tiny\ding{73}}}_3}
\newcommand{\newQfourstar}{{\Q}^{\mbox{\tiny\ding{73}}}_4}
\newcommand{\newQfivestar}{{\Q}^{\mbox{\tiny\ding{73}}}_5}
\newcommand{\newQtwostar}{{\Q}^{\mbox{\tiny\ding{73}}}_2}

\newcommand{\Nthreestar}{{\N}^{\mbox{\tiny\ding{73}}}_3}

\newcommand{\Nfourstar}{{\N}^{\mbox{\tiny\ding{73}}}_4}

\newcommand{\Ntwostar}{{\N}^{\mbox{\tiny\ding{73}}}_2}

\newcommand{\sistar}{{\Sigma}_\infty^{\mbox{\tiny\ding{73}}}}
\newcommand{\sistarstar}{{\Sigma}_\infty^{\mbox{\tiny\ding{72}}}}

\renewcommand{\star}{{^{\mbox{\tiny\ding{73}}}}}
\newcommand{\blackstar}{{^{\mbox{\tiny\ding{72}}}}}

\newcommand{\Kstar}{{\K^{\mbox{\tiny\ding{73}}}}}
\newcommand{\Kstarstar}{{\K^{\mbox{\tiny\ding{72}}}}}

\newcommand{\mstar}{{m^{\mbox{\tiny\ding{73}}}}}

\newcommand{\SC}{\mathscr S}

\newcommand{\IBB}{\I_{\mbox{\scriptsize\sf BB}}}
\newcommand{\IB}{\I_{\mbox{\scriptsize\sf Bose}}}

\newcommand{\Label}{\label}

\newcommand\conjb[1]{#1^{\mathsf c_{b}} }

\newcommand\conjpio[1]{{  #1^{{{\mathsf c}}}}}

\newcommand\conjB[1]{{  #1^{\mathsf c_{\scalebox{0.55}{\mbox{$\pi$}}}}}}

\begin{document}
%

\title{Specialness and  the Bose representation}

\author{S.G. Barwick, Wen-Ai Jackson and Peter Wild
\\ School of Mathematics, University of Adelaide, Adelaide 5005, Australia
}
\date{\today}
\maketitle




%
%

\begin{abstract} This article looks at subconics of order $q$  of $\PG(2,q^2)$ and characterizes them in the Bruck-Bose representation in $\PG(4,q)$.  
In common with other objects in the Bruck-Bose representation, the
 characterisation  uses the transversals of the regular line spread $\S$ associated with  the Bruck-Bose representation.
  By working in the Bose representation  of $\PG(2,q^2)$  in $\PG(5,q)$, we give a geometric explanation as to why 
  the transversals of the regular spread $\S$ are intrinsic to the characterisation of varieties of $\PG(2,q^2)$. 
\end{abstract}

\section{Introduction}\Label{sec:intro}

The focus of this article is conics contained in a Baer subplane of $\PG(2,q^2)$, and the corresponding structure in the Bose representation in $\PG(5,q)$ and in the Bruck-Bose representation in $\PG(4,q)$. We define an $\Fqq$-conic in $\PG(2,q^2)$ to be a non-degenerate conic of $\PG(2,q^2)$.  We define an $\Fq$-conic of $\PG(2,q^2)$ to be a non-degenerate conic in a Baer subplane of $\PG(2,q^2)$. That is, an $\Fq$-conic is projectively equivalent to a set of points in $\PG(2,q)$ that satisfy a non-degenerate homogeneous quadratic equation over $\Fq$.  For the remainder of this article, $\bar \C$ will denote an $\Fq$-conic  in a Baer subplane $\bar \B$ of $\PG(2,q^2)$. Further, we always denote the unique $\Fqq$-conic containing $\bar \C$ by $\bar \Cplus$.

The Bruck-Bose representation of $\PG(2,q^2)$ in $\PG(4,q)$ employs a regular line spread $\S$ in a hyperplane $\si$.  The
interaction  of certain varieties of $\PG(4,q)$ with the transversal lines $g,g^q$  of $\S$  is intrinsic to their characterisation as varieties of $\PG(2,q^2)$. 
The following known characterisations use the transversal lines $g,g^q$, and we refer to these varieties as \emph{$g$-special.} 

\begin{result}\Label{g-spec-know}
\begin{enumerate} 
\item {\rm \cite{CasseQuinn2002}} A conic  $\N_2$ in $\PG(4,q)$ is called $g$-special if the quadratic extension of $\N_2$  to $\PG(4,q^2)$ meets $g$ in one point. The $g$-special conics of $\PG(4,q)$ correspond precisely to  
the  Baer sublines of $\PG(2,q^2)$ disjoint from $\li$.
\item {\rm \cite{CasseQuinn2002}} 
  A ruled cubic surface $\V^3_2$ in $\PG(4,q)$  is called $g$-special if the quadratic extension of $\V^3_2$  to $\PG(4,q^2)$  contains $g$ and $g^q$. The $g$-special ruled cubic surfaces in $\PG(4,q)$ correspond precisely to the 
 Baer subplanes of $\PG(2,q^2)$ tangent to $\li$. 
 \item {\rm \cite{metsch}}
 An orthogonal cone $\U$  in $\PG(4,q)$   is called $g$-special if the quadratic extension of $\U$  to $\PG(4,q^2)$  contains $g$ and $g^q$. The $g$-special orthogonal cones  in $\PG(4,q)$ correspond precisely to the 
classical unitals of $\PG(2,q^2)$.
\item {\rm 
 \cite{BJW}}
 A 3 or 4-dimensional  normal rational curve $\N$ in $\PG(4,q)$  is called $g$-special  if 
  the quadratic extension of $\N$  to $\PG(4,q^2)$ meets $g$ in one or two points respectively. 
  The $g$-special 3 and 4-dimensional  normal rational curves 
   in $\PG(4,q)$ correspond precisely to the 
 $\Fq$-conics in a tangent Baer subplane  of $\PG(2,q^2)$.
  \end{enumerate}
  \end{result}

In this article, we investigate the notion of $g$-special varieties in $\PG(4,q)$  by working in the Bose representation of $\PG(2,q^2)$ in $\PG(5,q)$. This perspective allows us to give a geometric explanation as to why varieties of $\PG(2,q^2)$ give rise to $g$-special varieties of $\PG(4,q)$. 

The article is set out as follows. Section 2 discusses the Bruck-Bose and the Bose representations of $\PG(2,q^2)$ and their relationship to each other,  introducing the notation that we use. Section~\ref{sec:scroll} introduces  the notion of a generalised scroll. In Section~\ref{sec:Baer}, we look at Baer sublines, Baer subplanes, conics and $\Fq$-conics of $\PG(2,q^2)$. We determine their representation in the $\PG(5,q)$ Bose setting, and determine the  extension of these objects to $\PG(5,q^2)$. 

Our main focus for the remainder of the article  is using geometric arguments to characterise the representation of $\Fq$-conics of $\PG(2,q^2)$ in the $\PG(4,q)$ Bruck-Bose setting. 
In Theorem~\ref{conic-subplane-A}, we determine the representation of $\Fq$-conics of $\PG(2,q^2)$ in the $\PG(5,q)$ Bose representation. In Theorem~\ref{conic-quartic}, we use the Bose representation to  give a  short coordinate-free proof that $\Fq$-conics  of $\PG(2,q^2)$  correspond  to a quartic curve in the  $\PG(4,q)$ Bruck-Bose setting; there are five different cases which are described in Corollary~\ref{cor:gspecial}.

In Section~\ref{sec-2special},  we refine the definition of a $g$-special normal rational curve to define a 2-special normal rational curve. We use this to give a short, unified characterisation of all $\Fq$-conics in $\PG(2,q^2)$, namely  the 2-special normal rational curves of $\PG(4,q)$ correspond precisely to the $\Fq$-conics of $\PG(2,q^2)$.

\section{Background}

%
%
%

\subsection{Quadratic extension of varieties}

A variety in $\PG(n,q)$ has a natural extension to a variety in $\PG(n,q^2)$ and in $\PG(n,q^4)$,  we describe the notation we use. 
If $\K$ is
 variety of $\PG(n,q)$, 
 then $\K$ is the set of points of $\PG(n,q)$ satisfying a set of $k$ homogeneous $\Fq$-equations   $\F=\{f_i(x_0,\ldots,x_n)=0, \ i=1,\ldots,k\}$. 
We define the  {\em variety-extension} of $\K$ to $\PG(n,q^2)$, denoted  $\Kstar$, to be the variety consisting of the set of points of $\PG(n,q^2)$ that  satisfy the same set $\F$  of homogeneous equations as 
$\K$.  Similarly, we can define the {\em variety-extension} of $\K$ to $\PG(n,q^4)$, denoted  $\Kstarstar$.
So if $\Pi_r$ is an $r$-dimensional subspace of $\PG(n,q)$, then $\Pirstar$ is the natural extension to an $r$-dimensional subspace of $\PG(n,q^2)$, and $\Pirstarstar$ is the extension to $\PG(n,q^4)$. In this article, we use the $\star,\blackstar$ notation for $n=3,4,5$. We do not use it for varieties in $\PG(2,q^2)$.

    \subsection{The Bruck-Bose representation of $\PG(2,q^2)$}\Label{sec:intro-BB}
 
The Bruck-Bose representation of $\PG(2,q^2)$ in $\PG(4,q)$ was introduced independently by Andr\'e \cite{andr54} and Bruck and Bose \cite{bruc64,bruc66}. 
Let $\si$ be a hyperplane of $\PG(4,q)$ and let $\S$ be a regular spread
of $\si$. Consider the  incidence
structure $\IBB$ with 
 {\sl points}  the points of $\PG(4,q)\backslash\si$ and the lines of $\S$;  {\sl lines}  the planes of $\PG(4,q)\backslash\si$ that contain
  an element of $\S$, and a line at infinity $\li$ whose points correspond to the lines of $\S$; and {\sl incidence} induced by incidence in
  $\PG(4,q)$. Then $\IBB\cong\PG(2,q^2)$.
We call $\IBB$ the {\em Bruck-Bose representation} of $\PG(2,q^2)$ in $\PG(4,q)$.  If $\bar \K$ is a set of points in $\PG(2,q^2)$, then we denote the corresponding set of points  in $\IBB$ by $[\K]$.  
Associated with a regular spread $\S$ in $\PG(3,q)$ are a unique pair of {\em transversal lines} in the quadratic extension $\PG(3,q^2)$. These transversal lines are disjoint from $\PG(3,q)$, and are called conjugate with respect to the map $X=(x_0,x_1,x_2,x_3)\mapsto X^q=(x_0^q,x_1^q,x_2^q,x_3^q)$. We denote these transversal lines by $g,g^q$. The spread $\S$ is the set of $q^2+1$ lines $XX^q\cap\PG(3,q)$ for $X\in g$, in particular, the points of $\li$ are in one-to-one correspondence with the points of $g$.  For more details on this representation, see \cite{UnitalBook}.

\subsection{The Bose representation of $\PG(2,q^2)$}\Label{sec:intro-Bose}

\subsubsection{Geometric construction}

Bose \cite{Bose} introduced the following representation of $\PG(2,q^2)$ as a line spread  in $\PG(5,q)$.  Embed $\PG(5,q)$ in $\PG(5,q^2)$ and
let $\Pit$ be a plane in $\PG(5,q^2)$ which is disjoint from $\PG(5,q)$.
There is a unique involutory automorphism that fixes $\PG(5,q)$ pointwise, namely $X=(x_0,\ldots,x_5)\mapsto X^q=(x_0^q,\ldots,x_5^q)$, we call $X,X^q$ \emph{conjugate points} of $\PG(5,q^2)$.  The plane $\Pit$ has a (disjoint) \emph{conjugate plane} $\Pit^q$. 

We define an incidence structure  $\IB$. The  {\em points} of $\IB$ are the $q^4+q^2+1$  lines of $\PG(5,q)$ of form $XX^q\cap\PG(5,q)$ for points 
 $X\in\Pit$. These lines form a regular 1-spread of $\PG(5,q)$ denoted by ${\mathbb S}$, and we call $\Pit$, $\Pit^q$ 
 the {\em transversal planes} of ${\mathbb S}$. Conversely, any regular  1-spread of $\PG(5,q)$ has a unique set of transversal planes in $\PG(5,q^2)$, and can be constructed in this way. The  {\em lines} of $\IB$   are the 3-spaces of $\PG(5,q)$ that meet ${\mathbb S}$ in $q^2+1$ lines. A straightforward counting argument shows that 
 these 3-spaces form a dual spread $\mathbb H$ (that is, each 4-space of $\PG(5,q)$ contains a unique 3-space in $\mathbb H$).  {\em Incidence} in $\IB$ is inclusion.
  Then $\IB\cong\Pit\cong\PG(2,q^2)$, and $\IB$  is called the {\em Bose representation} of $\PG(2,q^2)$ in $\PG(5,q)$. 
  
\subsubsection{Notation}
  We use the following notation. A point $\bar X$ in $\PG(2,q^2)$ corresponds to a unique point of the transversal plane $\Pit$ denoted $X$, and the Bose representation of  $\bar X$  is the line of $\mathbb S$ denoted by $\Bo{X}=XX^q\cap\PG(5,q)$. Note that $$\Bo{X}\star=XX^q \quad \textup{and} \quad \Bo{X}\star\cap\Pit=X.$$ 
 More generally, if $\bar\K$ is a set of points of $\PG(2,q^2)$, then $\K$ denotes the corresponding set of points of  the transversal plane $\Pit$ and $\lbose \K\rbose$ denotes the corresponding set of lines in the Bose representation in $\PG(5,q)$. 

\subsubsection{Baer sublines and subplanes in the Bose representation}
 The next result of Bose \cite{Bose}  describes the representation of Baer sublines and subplanes of $\PG(2,q^2)$ in $\PG(5,q)$. See \cite{HT} for details on the Segre variety $\S_{1;2}$. 
  
\begin{result}\Label{Bose-Baer}
\begin{enumerate}
\item Let $\bar \bs$ be a Baer subline of $\PG(2,q^2)$, then in $\PG(5,q)$, $\Bo{\bs}$ is a regulus.
\item   Let $\bar \B$ be a Baer subplane of $\PG(2,q^2)$, then in $\PG(5,q)$, the lines of $\lbose\B\rbose$ form the maximal systems of lines of a  Segre variety $\S_{1;2}$. 
\end{enumerate}
\end{result}

 \subsubsection{Coordinates in the $\PG(5,q)$ Bose representation}\Label{sec:coord}

We present coordinates for the Bose representation of $\PG(2,q^2)$ in $\PG(5,q)$.  Let $\tau$ be a primitive element of $\Fq$ with primitive polynomial $$x^2-t_1x-t_0$$ for $t_0,t_1\in\Fq$, so $\tau+\tau^q=t_1$ and $\tau\tau^q=-t_0$. 
 Let $\bar P=(x,y,z)\in\PG(2,q^2)$, so we can write $$\bar P=\big(x_0+x_1\tau,\ y_0+y_1\tau,\ z_0+z_1\tau\big)$$ for unique $x_i,y_i,z_i\in\Fq$. 
 
 \begin{lemma}
Let $\bar P\in \PG(2,q^2)$ have coordinates $\bar P=(x,y,z)=\big(x_0+x_1\tau,\ y_0+y_1\tau,\ z_0+z_1\tau\big)$ as above. Then the Bose representation of $\bar P$ in $\PG(5,q)$ is the line  $\Bo{P}=P_0P_1$ with 
\begin{eqnarray*}
 P_0&=&(x_0,x_1,y_0,y_1,z_0,z_1)\\
 P_1&=&\Big(x_1t_0,\ x_0+x_1t_1,\ y_1t_0,\ y_0+y_1t_1,\ z_1t_0,\ z_0+z_1t_1\Big).
 \end{eqnarray*}
\end{lemma}

\begin{proof} The point $\bar P$ has homogeneous coordinates $(x,y,z)\equiv \rho (x,y,z)$ for any $\rho\in\Fqqq\setminus \{0\}$. 
As $\rho$ varies,  we generate a related point of $\PG(5,q)$, giving us the $(q^2-1)/(q-1)=q+1$ points of the line $\Bo{P}$ of the Bose spread. We determine two of these points to determine the line. 
Firstly, related to the coordinate representation of $\bar P=\big(x_0+x_1\tau,\ y_0+y_1\tau,\ z_0+z_1\tau\big)$ with $\rho=1$ is the point $P_0$ of $\PG(5,q)$ with coordinates given above. 
 Secondly,  consider the representation of $\bar P$ with $\rho=\tau$, that is, $\bar P=\tau(x,y,z)=(\tau x,\tau y,\tau z)$. The first coordinate of this expands as 
$\tau x=\tau(x_0+x_1\tau)=x_1t_0+\tau(x_0+x_1t_1)$. The second and third coordinates expand similarly. Corresponding to this coordinatate representation of $\bar P$ is the point 
 $P_1\in\PG(5,q)$  given above. So the Bose representation of $\bar P$ in $\PG(5,q)$ is the line $\Bo{P}=P_0P_1.$
\end{proof}

Next we determine in $\PG(5,q^2)$ the coordinates of the two unique transversal planes $\Pit$, $\Pit^q$ of the Bose regular 1-spread $\mathbb S$.

\begin{lemma}\label{coordPit}
The transversal planes of the Bose regular 1-spread $\mathbb S$ are
$$\Pit=\langle A_1,A_2,A_3\rangle \quad \mbox{and}\quad \Pit^q=\langle A_1^q,A_2^q,A_3^q\rangle$$
where $$A_1=(\tau^q,-1,\,  0,0,\, 0,0),\ \quad
A_2=(0,0,\, \tau^q,-1,\,  0,0),\  \quad
A_3=(0,0,\, 0,0,\, \tau^q,-1).$$
\end{lemma}

\begin{proof} In $\PG(5,q^2)$, let 
$A_1$, 
$A_2$, $A_3$ be the points given in the statement, and
 let $\alpha$ be the plane $\alpha=\langle A_1,A_2,A_3\rangle$.
Observe that the plane $\alpha$ lies in $\PG(5,q^2)\setminus\PG(5,q)$. 
We show that the extension of every  line  of the Bose spread $\mathbb S$ meets the plane $\alpha$ in a point, and hence $\alpha=\Pit$. 
First observe   the triangle of reference of $\PG(2,q^2)$ which is $\bar X=(1,0,0)$, $\bar Y=(0,1,0)$, $\bar Z=(0,0,1)$. Correspondingly in $\PG(5,q^2)$, we have $$\Bo{X}\star\cap \alpha=A_1,\quad \Bo{Y}\star\cap\alpha=A_2,\quad \Bo{Z}\star\cap\alpha=A_3.$$
Now consider the point $\bar P=(x,y,z)$ of $\PG(2,q^2)$. Then $\Bo{P}\star$ meets $\alpha$ in the point $P$ where $$P\ =\ \tau^q P_0-P_1\ = \ xA_1+yA_2+zA_3.$$ Hence each line of the Bose spread  $\mathbb S$ meets $\alpha$ in a point, and so $\alpha$ is a transversal plane of $\mathbb S$. That is, the two transversal planes of the regular spread $\mathbb S$ are $\Pit=\langle A_1,A_2,A_3\rangle$, and $\Pit^q=\langle A_1^q,A_2^q,A_3^q\rangle.$
\end{proof}

Now consider the Baer subplane $\bar\pi_0=\PG(2,q)$ of $\PG(2,q^2)$. 
Let $\bar{\mathsf c}$ be the unique involution acting on points of $\PG(2,q^2)$ that fixes $\bar\pi_0$ pointwise.
In $\PG(2,q^2)$, we have  $$\bar P= (x,y,z) \ \ \longmapsto\ \  {\bar {P}}^{\bar{\mathsf c}}=(x^q,y^q,z^q).$$ In the transversal plane $\Pit\subset\PG(5,q^2)$, correspondingly we have the map $\mathsf c$ which acts on points of 
$\Pit$ and fixes the Baer subplane  $\pi_0$ pointwise, where $$P=xA_0+yA_1+zA_2\ \ \longmapsto \ \   {\conjpio{P}}=x^qA_0+y^qA_1+z^qA_2\ \in\ \Pit.$$

In summary, corresponding to the point $\bar P=(x,y,z)\in\PG(2,q^2)$, we have four related points on the transversal planes  in $\PG(5,q^2)$ 
$$
\begin{array}{rcrcrcrclclcl}
 P&=&xA_1&+&yA_2&+&zA_3& \in& \Pit\\
  {\conjpio{P}}&=&x^qA_1&+&y^qA_2&+&z^qA_3& \in&\Pit\\
P^q&=&x^qA_1^q&+&y^qA_2^q&+&z^qA_3^q& \in& \Pit^q\\
(\conjpio{P})^q&=&xA_1^q&+&yA_2^q&+&zA_3^q & \in & \Pit^q.
\end{array}
$$

In the next section we demonstrate how the Bruck-Bose setting is embedded in the Bose setting. Note that writing $\PG(5,q)=\{(x_0,\ldots,x_5)\,|\,x_i\in\Fq\}$, and intersecting the coordinates  for $\IB$ described above with the 4-space $\Pi_g$ of equation $x_5=0$ gives the coordinates for the Bruck-Bose representation of $\PG(2,q^2)$ in $\Pi_g=\PG(4,q)$ in the format described in \cite{UnitalBook}.

    \subsection{The Bruck-Bose representation inside the Bose representation}\Label{sec:intro-BBinB}
 
 We can construct the Bruck-Bose representation $\IBB$ of $\PG(2,q^2)$ as a subset of the Bose representation $\IB$ by essentially intersecting a 4-space with the Bose representation  in $\PG(5,q)$. 
 Let $\Pi_g$ be a 4-space of $\PG(5,q)$, so $\Pi_g$ contains a unique 3-space of $\IB$, which we denote by $\si$. The extension of $\Pi_g$ to $\PG(5,q^2)$ meets the transversal plane $\Pit$ of the Bose spread $\mathbb S$ in a line denoted $g$. Further, $\langle g,g^q\rangle\cap\PG(5,q)=\si$. The intersection of   the 4-space $\Pi_g$ which an element of the Bose representation $\IB$ gives the corresponding element of the Bruck-Bose representation. That is, if $\bar \K$ is a subset of $\PG(2,q^2)$, then $\Bo{K}\cap\Pi_g=[\K]$, and we write $\IBB=\IB\cap\Pi_g$, see Figure~\ref{fig1}.  
So we have the following correspondences:
 $$
\begin{array}{ccccccccccc}
\PG(2,q^2)&\cong&\Pit&\cong&\IB&\cong&\IBB\\
\bar P &\longleftrightarrow&
P
&\longleftrightarrow&
\Bo{P}=PP^q\cap\PG(5,q)
&\longleftrightarrow&
[P]=\Bo{P}\cap\Pi_g.
\end{array}$$
The lines $g,g^q$ are the transversal lines of the Bruck-Bose regular 1-spread $\S$ in $\si$. So $\li$ in $\PG(2,q^2)$ corresponds to the line $g$ in  the transversal plane $\Pit$, and to the spread $\S$ in the 3-space  $\si=\langle g,g^q\rangle\cap\PG(5,q)$, where $\S=\mathbb S\cap \si$. 
For a point $\bar A\in\li$, we have a corresponding point  $A\in g$ and corresponding spread line  $\Bo{A}=[A]$. 
  \begin{figure}[h]\caption{Bruck-Bose inside Bose}\label{fig1}
 \centering
 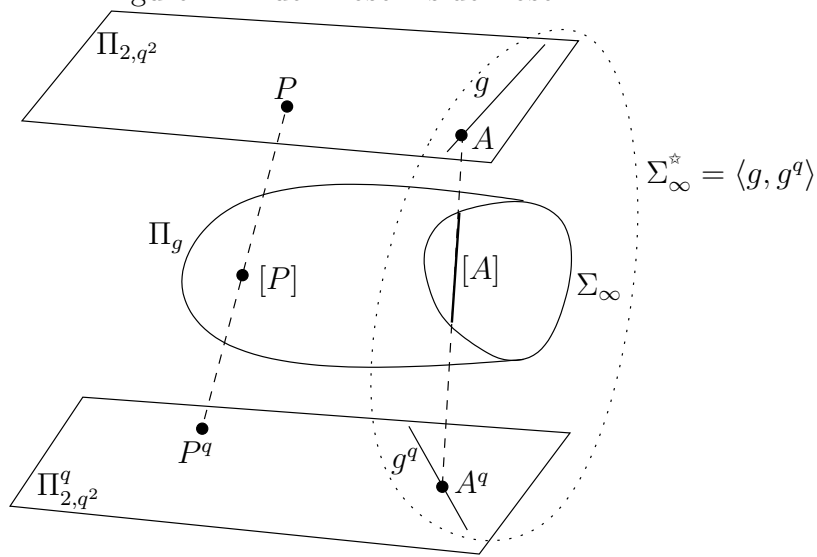
 \end{figure}

\begin{remark}\Label{remark:BBexact}
{\rm 
The Bruck-Bose correspondences described in the literature are not always exact-at-infinity. That is, we do not always include lines contained in $\si$ in our description. However, when considering  the Bruck-Bose representation as a subset of the Bose representation, that is, $\IBB=\IB\cap \Pi_g$, then we get a representation which is  exact-at-infinity. 
 For example,  we usually say that a Baer subplane $\bar \B$ of $\PG(2,q^2)$ secant to $\li$ corresponds to a plane $[\B]$ of $\PG(4,q)$ that does not contain a spread line (see \cite[Theorem 3.13]{UnitalBook}). To give a representation which is  exact on $\si$, we need to describe $[\B]$ as a plane $\alpha$ that does not contain a spread line, {\em together with} the $q+1$ spread lines that meets $\alpha$. We distinguish between these conventions by using   the phrase \emph{``the exact-at-infinity} Bruck-Bose representation in $\PG(4,q)$''. 
 }\end{remark}

 %
%
%
%

\subsection{Notation summary}
 
 \begin{itemize}
 \item $\bar P,\ \bar \C,\ \bar\B,\ldots$ are used to represent objects in $\PG(2,q^2)$. 
  \item For a variety $\K$ in $\PG(n,q)$, $n>2$, we denote the extension to $\PG(n,q^2)$ by $\Kstar$, and the extension to $\PG(n,q^4)$ by $\K\blackstar$. 
  \item In $\PG(n,q^h)$, $n>2$, we let $X^q$ denote the map $X=(x_0,\ldots,x_n)\longmapsto X^q=(x^q_0,\ldots,x^q_n)$.
 \item $\bar \C$ denotes an $\Fq$-conic in $\PG(2,q^2)$, and $\bar \Cplus$ denotes the unique $\Fqq$-conic  in $\PG(2,q^2)$ containing $\bar \C$.
 \item In $\IB$
 \begin{itemize}
 \item[$\cdot$]   $\mathbb S$ is a regular 1-spread in $\PG(5,q)$.
  \item[$\cdot$]   $\mathbb S$ has two transversal planes in $\PG(5,q^2)$, denoted $\Pit$, $\Pit^q$. 
 \item[$\cdot$] 
 A point $\bar P$ in $\PG(2,q^2)$ corresponds to a point $P$ in the transversal plane $\Pit$, and to a line $\Bo{P}$ of $\mathbb S$, where $\Bo{P}=PP^q\cap\PG(5,q)$.

 \item[$\cdot$] For a Baer subline $\bs$ contained in a line $\ell_b$ of $\Pit$, the unique involution acting on points of $\ell_b$ that fixes $\bs$ pointwise is denoted $\mathsf c_b$, and called \emph{conjugation with respect to $\bs$.} 
   \item[$\cdot$] For a Baer subplane $\B$ contained in $\Pit$, the unique involution acting on points of $\Pit$ that fixes $\B$ pointwise is denoted $\mathsf c_\pi$, and called \emph{conjugation with respect to $\B$}. Note that for a Baer subline $\bs$ contained in a line $\ell_b$ of $\Pit$, the maps $\mathsf c_\pi$ and $\mathsf c_b$ agree on $\ell_b$ iff $b$ is a line of $\B$. 
   \end{itemize}
  \item In $\IBB$
 \begin{itemize}
\item[$\cdot$]   $\S$ is a regular 1-spread in the 3-space at infinity $\si\cong\PG(3,q)$. 
\item[$\cdot$]  $\S$ has transversal lines denoted 
 $g,g^{q}$ in $\sistar\cong\PG(3,q^2)$.
\item[$\cdot$]   For a point $\bar X\in\PG(2,q^2)\setminus \li$, we denote corresponding point of $\PG(4,q)\setminus\si$ by  $[X]$.
\item[$\cdot$]   If $\bar X\in\li$, then we denote the corresponding point on $g$ by $X$ and the corresponding spread line by $[X]$, so $[X]=XX^q \cap\si$ and $X=[X]\star\cap g$. 
  \end{itemize}
 \item  in $\IBB=\IB\cap\Pi_g$, where $\Pi_g$ is a 4-space of $\PG(5,q)$. 
\begin{itemize}
\item[$\cdot$]   $\Pigstar$ meets  the transversal plane $\Pit$ in the line $g$  which is the transversal line of the regular spread $\S=\{PP^q\cap\Pi_g\st P\in g\}$.
\item[$\cdot$]  For a point $\bar X\in\PG(2,q^2)$, we have $[X]=\Bo{X}\cap\Pi_g$.
\end{itemize}
\end{itemize}

 \section{Generalising scrolls}\Label{sec:scroll}

We begin by  defining the notion of a scroll that rules two normal rational curves according to a projectivity, see for example \cite{vande}. For details on normal rational curves, see \cite{HT}. In $\PG(n,q)$, let $\N_r$ be an $r$-dimensional normal rational curve contained in an $r$-space $\Sigma_r$, and let $\N_s$ be an $s$-dimensional normal rational curve contained in an $s$-space $\Sigma_s$, where $\Sigma_r\cap\Sigma_s=\emptyset$. Denote the parameters of $\N_r$, $\N_s$ by $\theta$, $\phi$ respectively. That is,  there is a homography in $\PGL(n+1,q)$ acting on the points of $\PG(n,q)$ that maps $\N_r$ to the set $\{R_\theta=(1,\theta,\ldots,\theta^r,\ 0,\ldots,0)\,|\, \theta\in\Fq\cup\{\infty\}\} $ and  maps $\N_s$ to the set $\{S_\phi=(0,\ldots,0,\ 1,\phi,\ldots,\phi^s,)\,|\, \phi\in\Fq\cup\{\infty\}\} $.
Let $\sigma\in\PGL(2,q)$ be a projectivity (homography) that maps $(1,\theta)$ to $(1,\phi)$. The 
projectivity $\sigma$ determines a bijection map from the points of $\N_r$ to the points of $\N_s$. The set of $q+1$ lines of $\PG(n,q)$ that join each point of $\N_r$ with the corresponding point of $\N_s$ is called a \emph{scroll}, denoted $\scroll(\N_r,\N_s,\sigma)$. We say that  $\scroll(\N_r,\N_s,\sigma)$ is \emph{a scroll that rules $\N_r$ and $\N_s$ according to the projectivity $\sigma\in\PGL(2,q)$. }

For example,  a regulus $\R$ of $\PG(3,q)$ is a scroll as follows.
Let $\ell,m$ be two lines of opposite regulus  of $\R$, then the lines of $\R$ are generated by a projectivity of $\PGL(2,q)$
ruling $\ell$ and $m$, that is, the lines of $\R$ form a scroll. 
Another example is  the Bruck-Bose representation in $\PG(4,q)$ of a tangent Baer subplane of $\PG(2,q^2)$ as in Result~\ref{g-spec-know}(1). This  ruled cubic surface $\V^3_2$  rules a line and a conic according to a projectivity of $\PGL(2,q)$, and so is a scroll, see~\cite{UnitalBook}.

We generalise the notion of a scroll ruling two normal rational curves   to scrolls which rule two varieties contained in disjoint planes of $\PG(5,q)$. 
For example, a Segre variety $\S_{1;2}$ is a set of  planes $\alpha_1,\ldots,\alpha_{q+1}$, which are ruled by  lines $\ell_1,\ldots,\ell_{q^2+q+1}$, see \cite{HT}. The points of the plane $\alpha_i$ can be thought of as $\{(x_i,y_i,z_i)\,|\,x_i,y_i,z_i\in\Fq,\ \textup{not all zero}\}$.  The lines $\ell_1,\ldots,\ell_{q^2+q+1}$ rule the two planes $\alpha_1,\alpha_2$ according to a homography $\sigma\in\PGL(3,q)$. That it, $\sigma$   maps the  coordinates of points in $\alpha_1$ (written in the form $(x_1,y_1,z_1)$) to  the coordinates of points in $\alpha_2$  (written  in the form  $(x_2,y_2,z_2)$), and if $\sigma(x_1,y_1,z_1)=(x_2,y_2,z_2)$, then one of the ruling lines $\ell_j$ contains the point  of $\alpha_1$ corresponding to $(x_1,y_1,z_1)$  and the point of $\alpha_2$ corresponding to  $(x_2,y_2,z_2)$.
We call this set of lines a scroll $\scroll(\alpha_1,\alpha_2,\sigma)$, as it rules two planes (which are varieties with two non-homogeneous coordinates) according to a homography in $\PGL(3,q)$.

Finally, we generalise this notion of scrolls to rulings of 
 two Baer subspaces. 
In $\PG(3,q^2)$, let $\bs_1$ be a Baer subline of a line $\ell_1$, and let  $\bs_2$ be a Baer subline of a line $\ell_2$, with $\ell_1\cap\ell_2=\emptyset$. The coordinates of the points of $\bs_1$ can be written as $(1,\theta)$, $\theta\in\Fq\cup\{\infty\}$, 
and the coordinates of the points of $\bs_2$ can be written as $(1,\phi)$, $\phi\in\Fq\cup\{\infty\}$. Let $\sigma\in\PGL(2,q)$ be a projectivity, such that $\sigma\colon(1,\theta)\mapsto(1,\phi)$. 
Then the set of lines that joins each point of $\bs_1$ to the corresponding (under $\sigma$) point of $\bs_2$ 
 form 
a scroll denoted $\scroll(\bs_1,\bs_2,\sigma)$.

Similarly, in $\PG(5,q^2)$ let $\B_1$ be a Baer subplane of a plane $\alpha_1$, and $\B_2$ a Baer subplane of the plane $\alpha_2$ with $\alpha_1\cap\alpha_2=\emptyset$. 
 Then there is a homography $\sigma\in\PGL(3,q)$ that maps the  coordinates of points in $\B_1$ (written in the form $(x_1,y_1,z_1)$) to  the coordinates of points in $\B_2$  (written  in the form  $(x_2,y_2,z_2)$). 
The lines of $\PG(5,q^2)$ that join points of $\B_1$ to the corresponding (via $\sigma$) points of $\B_2$ 
form a scroll 
 $\scroll(\B_1,\B_2,\sigma)$.

{\em Scroll-extensions:\ } We can naturally extend a scroll of $\PG(5,q)$ to a scroll of $\PG(5,q^2)$ as follows.  Note that  if $\U$ is a variety in a plane $\Pi_\U$ of $\PG(5,q)$, then the extension to $\PG(5,q^2)$, denoted $\U\star$, is the set of points in the extended plane $\PiUstar$ that satisfy the same set of equations that define $\U$.    
Let $\scroll(\U,\W,\sigma)$ be a scroll of $\PG(5,q)$, then we define the {\em scroll-extension} to $\PG(5,q^2)$ to be the scroll $\scroll(\U\star,\W\star,\sigma\star)$, where $\sigma$ acts over $\Fq$, and $\sigma\star$ is the natural extension   acting over $\Fqq$.  

We now consider a scroll of $\PG(5,q)$  which rules two conics according to a projectivity of $\PGL(2,q)$, and  show that the pointset of this scroll forms a variety. Further, we determine the order and dimension of the variety using techniques and concepts as given in Semple and Roth \cite[I.4]{SR}.

\begin{lemma}\Label{conic-v42} In $\PG(5,q)$, let $\Pi$, $\Pi'$ be two disjoint planes, and let $\C$, $\C'$ be non-degenerate conics in $\Pi,\Pi'$ respectively. Then the pointset of any scroll $\scroll(\C,\C',\phi)$ is the pointset of  a variety of dimension 2 and order 4.
\end{lemma}

\begin{proof} 
Denote the set of points on the scroll $\scroll(\C,\C',\phi)$ by $\SC$. 
 Without loss of generality, we can coordinatise so that $\phi$ is essentially the identity. That is, let  $\C=\{P_{r,s}=(r^2,rs,s^2,0,0,0)\st r,s\in\Fq,\mbox{\ not\ both\ } 0\}$,  $\C'=\{P'_{r,s}=(0,0,0,r^2,rs,s^2)\st r,s\in\Fq,\mbox{\ not\ both\ } 0\}$, and 
$\SC$ consists of the points on the $q+1$ lines joining the point $P_{r,s}$ with the point $P'_{r,s}$ for $r,s\in\Fq,\mbox{\ not\ both\ } 0$.

We first  show that the points of $\SC$ form a variety $\V(\SC)$ of  dimension 2 by showing that the points of $\SC$ are in \mbox{one-to-one} algebraic correspondence with the points of a plane in 
$\PG(3,q)$.  Following the notation of \cite[page 14]{SR}, consider the 
linear equation
$M(y_0,y_1,y_2,y_3)=y_3$, and the 
six cubics  with homogeneous equations
$F_0(y_0,y_1,y_2,y_3)=y_0^3$, $F_1(y_0,y_1,y_2,y_3)=y_0^2y_1$, $F_2(y_0,y_1,y_2,y_3)=y_0y_1^2$, $F_3(y_0,y_1,y_2,y_3)=y_0^2y_2$, $F_4(y_0,y_1,y_2,y_3)=y_0y_1y_2$, $F_5(y_0,y_1,y_2,y_3)=y_1^2y_2.$
Now $M=0$ is the equation of a plane $M_2$ of $\PG(3,q)$, and so   is an irreducible primal of dimension 2 and order 1. 
Consider the point $\big(F_0(y_0,y_1,y_2,0),\ldots, F_5(y_0,y_1,y_2,0)\big)$ in $\PG(5,q)$, this is a point of $\SC$ since it has form 
$$(y_0^3,\ y_0^2y_1,\ y_0y_1^2,\ y_0^2y_2,\ y_0y_1y_2,\ y_1^2y_2)= y_0(y_0^2,\ y_0y_1,\ y_1^2,0,0,0)\ +\ y_2( 0,0,0,y_0^2,\ y_0y_1,\ y_1^2).
$$
This is a one-to-one algebraic correspondence between the points on the plane $M_2$  and the points of $\SC$ (so that the generic point of $\SC$ arises from only one point of $M_2$). 
 Hence the pointset of   $\SC$  coincides with the pointset of a variety  denoted $\V(\SC)$  which has  dimension 2. Note that the map $F\colon M_2\mapsto \SC$ defined by $F(y_0,y_1,y_2)=\big(F_0(y_0,y_1,y_2,0),\ldots,\ F_5(y_0,y_1,y_2,0)\big)$ has kernel $\{(0,1,0,0),(0,0,1,0)\}$.
 
 Next we show that the variety $\V(\SC)$ has order $4$, that is, we show that  a generic 3-space of $\PG(5,q)$ meets $\SC$ in 4 points. Consider the 3-space of $\PG(5,q)$ which is the intersection of  two hyperplanes $\Pi,\Pi'$ with equations $g=a_0x_0+\ldots+a_5x_5=0$ and $g'=b_0x_0+\ldots+b_5x_5=0$. Now $\Pi\cap\SC$ algebraically corresponds to the points in the plane  $M_2$   satisfying the cubic $\K_1$ of equation
$a_0y_0^3+a_1y_0^2y_1+a_2y_0y_1^2+ a_3y_2y_0^2+a_4y_2y_0y_1+a_5y_2y_1^2=0$. Note that $\K_1$ contains $\ker F$.
Straightforward calculations show that each line of $M_2$ through the point $(0,0,1,0)$ meets $\K_1$ twice at $(0,0,1,0)$. That is,   
 $(0,0,1,0)$ is a double point of $\K_1$. 
 Similarly $\Pi'\cap\SC$   algebraically corresponds to the points of a cubic $\K_2$ in the plane $M_2$. 
 The cubic  $\K_2$ contains $\ker F$, and $(0,0,1,0)$ is a double point of $\K_2$. 
 Two cubics of $\PG(2,q)$ meet generically in 9 points. As $(0,0,1,0)$ is a double point, it counts 4 times in this intersection. So in total, the points in $\ker F$ count 5 times in this intersection. Hence there are 4 points in $\K_1\cap\K_2$ which correspond to 4 points lying on $\SC$  in $\PG(5,q)$. Hence  
   $\SC$ meets  a generic 3-space in 4 points as required. 
\end{proof}

     \section{Bose varieties extended to $\PG(5,q^2)$}\Label{sec:Baer}

     We look at  the Bose representation of conics, Baer sublines, Baer subplanes and $\Fq$-conics of $\PG(2,q^2)$. We then investigate the extension of each variety from $\PG(5,q)$   to $\PG(5,q^2)$. 
     
 We begin with a lemma on lines that meet both transversal planes of the Bose spread $\mathbb S$.

\begin{lemma}\Label{lem:hypcong} Let $\ell,m$ be two  lines of $\PG(5,q^2)$ that meet both  the transversal planes $\Pit$  and $\Pit^q$. Then $\ell,m$  are either equal,  disjoint, or meet in  a point of $\Pit$  or $\Pit^q$. 
\end{lemma}

\begin{proof} Call a point  of $\PG(5,q^2)$ that lies in one of the transversal planes $\Pit$ or $\Pit^q$ a $T$-point. Call a line  of $\PG(5,q^2)$   that meets both   transversal planes a $T$-line. 
We prove the equivalent statement that each point of $\PG(5,q^2)$ not in $\Pit$ or $\Pit^q$ lies on a unique $T$-line. 
Let $P$ be a point    of $\PG(5,q^2)$ that is not a $T$-point. The 3-space $\langle P,\Pit\rangle$ meets $\Pit^q$ in a unique point $R$, and the 3-space $\langle P,\Pit^q\rangle$ meets $\Pit$ in a unique point $S$. Hence $P$ lies 
 on at least one $T$-line, namely the line $RS$. 
Suppose $P$  lies on two $T$-lines $\ell,m$. Then the four points $\ell\cap\Pit$, $m\cap\Pit$, $\ell\cap\Pit^q$, $m\cap\Pit^q$ are distinct. Hence $\langle \ell,m\rangle$ is a plane that meets $\Pit$ in a line and  meets $\Pit^q$ in a line. This contradicts $\Pit,\Pit^q$ being disjoint. Hence $P$ lies on at most one $T$-line, and we conclude that $P$ lies on exactly one $T$-line. 
\end{proof}

\subsection{Conics}

 \begin{theorem}\Label{Fqq-conic-C}
Let $\bar\O$ be a non-degenerate $\Fqq$-conic in $\PG(2,q^2)$, and consider the Bose representation of $\PG(2,q^2)$ in $\PG(5,q)$.
\begin{enumerate}
\item  In $\PG(5,q)$, the points of $\lbose \O\rbose$ form a variety $\V(\lbose \O\rbose)=\Q_1\cap\Q_2$ where $\Q_1,\Q_2$ are  two quadrics.
\item  In $\PG(5,q^2)$,  $\V(\lbose \O\rbose)\star=\newQonestar\cap\newQtwostar$, and the points of this variety consist of the points on the lines  
 $\{XY^q\st X, Y\in\O\}$. Hence $\V(\lbose \O\rbose)\star \cap\Pit=\O$.
\end{enumerate}
\end{theorem}

\begin{proof}
Without loss of generality we consider the non-degenerate conic $\bar\O$ of equation $y^2-zx=0$. We write $x=x_1+\tau x_2$, $y=y_1+\tau y_2$, $z=z_1+\tau z_2$, where $x_i,y_i,z_i\in\Fq$. Expanding and simplifying yields
$f_1+\tau f_2=0$ where $f_1=f_1(x_1,x_2,y_1,y_2,z_1,z_2)=y_1^2+t_0y_2^2-z_1x_1-t_0z_2x_2$ and $f_2=f_2(x_1,x_2,y_1,y_2,z_1,z_2)=t_1y_2^2+2y_1y_2-z_2x_1-z_1x_2-t_1z_2x_2=0$. The points of $\PG(5,q)$ satisfying $f_i=0$ form a quadric $\Q_i$, $i=1,2$. The set of points on the lines $\Bo{\O}$ is same as the set of points in the intersection of these quadrics. That is, the points of $\Bo{\O}$ coincide with the points of the  variety $\Q_1\cap\Q_2$, and we denote this variety by $\VO=\Q_1\cap\Q_2$.

In $\PG(5,q^2)$, the variety $\VO$ has extension $\VO\star=\newQonestar\cap\newQtwostar$. Moreover, $\VO\star$ is the intersection of any two distinct quadrics in the pencil $\newQonestar+\lambda\newQtwostar$, $\lambda\in\Fqq\cup\{\infty\}$. In particular, $\VO\star=\Q_3\cap\Q_4$ where $\Q_3,\Q_4$ are quadrics of $\PG(5,q^2)$ such that $\Q_3=\newQonestar+\tau^q \newQtwostar$ has equation $f_3=f_1+\tau^q f_2$ and  $\Q_4=\newQonestar+\tau \newQtwostar$ has equation $f_4=f_1+\tau f_2$.

Consider the quadric  $\Q_4=\{X\in\PG(5,q^2)\st f_4(X)=0\}$. 
Recall from Lemma~\ref{coordPit} that $\Pit=\langle A_0,A_1,A_2\rangle$. 
Note that as $f_4(A_2)\neq 0$, the point $A_2$ does not lie on quadric $\Q_4$, and so $\Q_4$ does not contain the transversal plane $\Pit$. Further, the point $P=A_1+\theta A_2+\theta^2 A_3$ satisfies $f_4(P)=0$, so lies on $\Q_4$ for $\theta\in\Fqq\cup\{\infty\}$. Hence $\Q_4\cap\Pit$ is a planar quadric which is not a plane, but contains the non-degenerate conic $\O=\{A_1+\theta A_2+\theta^2 A_3\,|\,\theta\in\Fq\cup\{\infty\}\}$. Thus $\Q_4\cap\Pit$ is the conic $\O$. 
The six partial derivatives of $f_4$ vanish at the points $A_1^q,A_2^q,A_3^q$. Hence the transversal plane $\Pit^q=\langle A_1^q,A_2^q,A_3^q\rangle$  lies in the singular space of $\Q_4$. As $\Pit$ is a plane that contains no line of $\Q_4$, it is disjoint from the singular space, and so the maximum dimension of the singular space of $\Q_4$ is two. Hence $\Q_4$ is a cone with vertex the transversal plane   $\Pit^q$ and base $\O$. 

Similarly the quadric  $\Q_3=\{X\in\PG(5,q^2)\st f_3(X)=0\}$ is a cone with vertex the transversal plane  $\Pit$ and base $\O^q$. 
That is, the points of $\Q_4$ are those on the set of all lines joining a point of $\Pit^q$ and a point of $\O$; and  the points of $\Q_3$ are those on  the set of all lines joining a point of $\Pit$ and a point of $\O^q$. By Lemma~\ref{lem:hypcong}, two  lines that meet both the transversal planes   $\Pit$  and $\Pit^q$ are either equal, disjoint, or meet in a point of $\Pit$  or $\Pit^q$. Hence the pointset of $\VO\star=\Q_3\cap\Q_4$ is ruled by the   lines joining a point of $\O$ and a point of $\O^q$, as stated in the theorem. 
\end{proof}

     \subsection{Baer sublines}
     
      Let $\bar \bs$ be a Baer subline of the line $\bar{\ell_b}$ in $\PG(2,q^2)$, and let $\bar{\mathsf c_b}$ be the unique involution in $\PGammaL(2,q)$ acting on the points of $\bar{\ell_b}$ which fixes $\bar \bs$ pointwise. In $\PG(5,q)$, this corresponds to the Baer subline $\bs$ of the line $\ell_b$ in  the transversal plane $\Pit$, and $\mathsf c_b$ is an involution which acts on the points of $\ell_b$, and fixes $\bs$ pointwise. Note that $\mathsf c_b$ is not a collineation of $\PG(5,q)$. If $X\in\ell_b$, we denote its \emph{conjugate point with respect to $\bs$} by $\conjb{X}\in\ell_b$.  The image of this point under the conjugacy map $X\mapsto X^q$ of $\PG(5,q^2)$ is the point $(\conjb{X})^q$, which  lies in the other transversal plane $\Pit^q$. In particular, we have four points of interest, $X,\conjb{X}\in\Pit$ and $X^q,(\conjb{X})^q\in\Pit^q$ 
 We will be interested in the line $X (\conjb{X})^q$ which we call a {\em scroll-line}. If $X\in b$, then $\conjb{X}=X$ and so the scroll-line becomes $X (\conjb{X})^q=XX^q$ which  meets $\PG(5,q)$ in the line $\Bo{X}$. If $X\in\ell_b\setminus b$, then the scroll-line $X (\conjb{X})^q$ is disjoint from  $\PG(5,q)$.

By Result~\ref{Bose-Baer}, the set of lines $\Bo{\bs}=\{\Bo{X}\st X\in b\}$ form a regulus of $\si\cong\PG(3,q)$, so they 
form one regulus of a  hyperbolic quadric $\Q$ of the 3-space $\Pi_b=\langle \ell_b,\ell_b^q\rangle\cap\PG(5,q)$.  So the pointset of $\Bo{\bs}$ coincides with the pointset of the variety $\Q$, and we denote this variety by $\Vb=\Q$. 
We carefully use this notation as the pointset of $\Bo{\bs}$ is  the pointset of more than one  variety. However, we are interested in the variety that relates to our usual Bruck-Bose representation of Baer sublines. In particular, we want to look at the variety-extension of $\Bo{\bs}$, and so we need to specify \emph{which} variety we mean.
The variety-extension of $\Vb=\Q$  to $\PG(5,q^2)$  is   $\Vb\star=\Q\star$.  The lines of the corresponding extended regulus were determined in 
\cite[Thm 5.3]{Bruck-construct} and consist of the lines we call  scroll-lines. 

\begin{result}\Label{subline-star} Let $\bs$ be a Baer subline of the line $\ell_b$ in  the transversal plane $\Pit$, so $\Bo{\bs}$ is a regulus in the 3-space $\Sigma_b=\langle \ell_b,\ell_b^q\rangle\cap\PG(5,q)$.  
Let
 $\Vb$ be the variety which is the hyperbolic quadric in $\Sigma_b$ whose points coincide with those of $\Bo{\bs}$. Then $
\Vb\star$ is a hyperbolic quadric in $\Sigmabstar$ with one regulus consisting of the  lines
$$ \{  X (\conjb{X})^q\st X\in \ell_b\}.$$ 
\end{result}

We noted in Section~\ref{sec:scroll} that a regulus in $\PG(3,q)$ is ruled by a projectivity, and so $\Bo{\bs}$ is a  scroll. As a projectivity is uniquely determined by the image of three points, the variety-extension and scroll-extension   of  $\Bo{\bs}$ coincide.  Hence the lines $ \{  X (\conjb{X})^q\st X\in \ell_b\}$ of $\Vb\star$ form a scroll.
Note that the set of lines $\{ XX^q\st X\in\ell_b\}$ in $\PG(5,q^2)$ do not  form a scroll  as $\Bo{\bs}$ determines a unique scroll of $\PG(5,q^2)$, namely the regulus $\{X(\conjb{X})^q\st X\in\ell_b\}$.


\subsection{Baer subplanes}
Next we consider a Baer subplane  of $\PG(2,q^2)$.  

\begin{result}\Label{result-HT}
Let $\bar\B$ be a Baer subplane of $\PG(2,q^2)$.
In $\PG(5,q)$,  the lines of $\lbose \B\rbose=\{XX^q\cap\PG(5,q)\st X\in\B\}$ form the maximal systems of lines of a Segre variety $\S_{1;2}$. Hence  the lines of $\lbose \B\rbose$ form   a scroll of $\PG(5,q)$, whose pointset  coincides with the points of a variety $\VB=\Q_1\cap\Q_2\cap\Q_3$ where $\Q_1,\Q_2,\Q_3$ are three quadrics of $\PG(5,q)$. 
\end{result}

\begin{proof}  Result~\ref{Bose-Baer} showed that the lines of $\lbose \B\rbose$  form the maximal systems of lines of a  Segre variety $\S_{1;2}$. Segre varieties are studied in \cite{HT}, and the points of $\S_{1;2}$ form the intersection of three quadrics. It
is straightforward to show that  
given any two planes $\alpha,\beta$ of 
$\S_{1;2}$, the ruling lines correspond to a homography (in the scroll sense)  between the points of $\alpha$ and $\beta$. 
 That is, $\lbose \B\rbose$ is a scroll of $\PG(5,q)$. 
\end{proof}

Corresponding to the Baer subplane $\B$ of  the transversal plane $\Pit$ is a conjugacy map acting on the points of $\Pit$ (namely the unique involutory automorphic collineation in $\PGammaL(3,q^2)$ acting on $\Pit$ that fixes $\B$ pointwise). If $X\in\Pit$, we denote its conjugate point with respect to $\B$ by $\conjB{X}$, noting that $\conjB{X}\in\Pit$.  The image of the point $\conjB{X}$ under the conjugate map $X\mapsto X^q$ of $\PG(5,q^2)$ is the point $(\conjB{X})^q$, which  lies in the other transversal plane $\Pit^q$. We call the line $X (\conjB{X})^q$ a {\em scroll-line}. If $X\in \B$, then the scroll-line becomes $X (\conjB{X})^q=XX^q$ which  meets $\PG(5,q)$ in the line $\Bo{X}$. If $X\in\Pit\setminus \B$, then $X (\conjB{X})^q\cap\PG(5,q)=\emptyset$.  It now follows from Results~\ref{subline-star} and \ref{result-HT} that 
 $\VB\star$ is a variety (namely the intersection of three quadrics) whose points lie on $q^4+q^2+1$ lines which each meet both transversal planes $\Pit$ and $\Pit^q$.

\begin{corollary}\Label{pi0-Bose} Let $\B$ be a Baer subplane of  the transversal plane $\Pit$, and  $\VB=\Q_1\cap\Q_2\cap\Q_3$. Then the points of $
\VB\star=\newQonestar\cap\newQtwostar\cap\newQthreestar$ are the points of a Segre variety  and the pointset of a  scroll. The ruling lines of $
\VB\star$ are $X(\conjB{X})^q$ for $ X\in\Pit$.
\end{corollary}

\begin{proof}
By Result~\ref{result-HT},  
$\VB=\Q_1\cap\Q_2\cap\Q_3$ and its points coincide with the points of a Segre variety $\S_{1;2}$. The variety-extension to $\PG(5,q^2)$ is $
\VB\star=\newQonestar\cap\newQtwostar\cap\newQthreestar$. As a homography is uniquely determined by an ordered quadrangle, the variety and scroll-extensions of  $\lbose \B\rbose$  to $\PG(5,q^2)$ coincide.
 Hence 
the points of  $\VB\star$ form a scroll, and   the Segre variety $\Sonetwostar$ of $\PG(5,q^2)$. 
  The ruling lines of this Segre variety  follow  directly from the Baer subline result given in Result~\ref{subline-star}. 
\end{proof}

\subsection{Interpreting $g$-special in the Bose representation}

We discuss how the Bose representation gives a geometric explanation of  the notion of $g$-special as defined in Result~\ref{g-spec-know}. That is, why many characterisation of varieties in the Bruck-Bose representation relate to the intersection with the transversal lines $g,g^q$ of the regular 1-spread $\S$ of $\si$.  

In $\PG(2,q^2)$, let  $\bar b$ be a Baer subline of the line $\bar \ell_b$, with $\bar \ell_b\cap\li=\bar B$ and $\bar B\notin \bar b$. 
We work in the Bose representation of $\PG(2,q^2)$ in $\PG(5,q)$. The line $\li$ corresponds to a line, denoted $g$, in the transversal plane $\Gamma$. Let   $\Pi_g$ be a 4-space of $\PG(5,q)$ whose extension contains the 3-space $\sistar=\langle g,g^q\rangle$. That is, $\Pigstar\cap\Pit=g$. We can consider the exact-at-infinity Bruck-Bose representation of $\PG(2,q^2)$ in $\Pi_g$ constructed as $\IBB=\IB\cap\Pi_g$. 
In the transversal plane $\Gamma$, $b$ is Baer subline of the line $\ell_b$, with $\ell_b\cap g=B$ and $B\notin b$. The Bose representation of $b$ is a regulus $\Bo{b}$, which is the pointset of a quadric (a variety)  denoted $\Vb$. The variety $\Vb$ lies in the 3-space $\Sigma_b=\langle b,b^q\rangle\cap\PG(5,q)$. As $B\notin b$, $\Pi_g\cap \Sigma_b$ is a plane, and $\Pi_g$ meets the regulus $\Bo{b}$ in a non-degenerate conic $\V([b])=\Vb\cap\Pi_g$, which is disjoint from $\si$. 
The variety-extension of the conic $\V([b])$ is the non-degenerate conic $\V([b])\star=\Vb\star\cap\Pigstar$. By Result~\ref{subline-star}, $\Vb\star$ meets $\Pit$ in  the line $\ell_b$, and so $\Vb\star\cap g=B$. Hence the extension of $[b]$ is a conic that meets $g$ in the point $B$. That is, the conic $[b]$ is $g$-special in the sense of Result~\ref{g-spec-know}(1).

 
 A similar observation gives a geometric interpretation in the Bose representation of  the other characterisations in Result~\ref{g-spec-know} of varieties in the Bruck-Bose representation in relation to the transversal lines $g,g^q$ of the regular 1-spread $\S$ in $\si$.

\subsection{$\Fq$-conics}

We now look at   an $\Fq$-conic $\C$ contained in the Baer subplane $\B$. We show that  in the Bose representation in $\PG(5,q)$, the lines of $\lbose \C\rbose$ form a   scroll, and the points form a variety $\VC$ which is the intersection of five quadrics. 
The variety-extension of $\VC$ is denoted $\VC\star$.
We show that when viewing $\lbose \C\rbose$ as a scroll, the pointset of the scroll-extension coincides with the pointset of $\VC\star$.
We use this to determine the ruling lines of  $\VC\star$, showing that they are related to $\Cplus$ (the unique $\Fqq$-conic of  the transversal plane $\Pit$ containing $\C$). The ruling lines of $\VC\star$ are illustrated in  Figure~\ref{fig2}.

\begin{theorem}\Label{conic-subplane-A}
Let $\C$ be an $\Fq$-conic in the Baer subplane $\B$ of  the transversal plane $\Pit$, then 
\begin{enumerate}
\item In $\PG(5,q)$, the lines of
 $\lbose \C\rbose$ form a scroll, and the pointset of
 $\lbose \C\rbose$ forms a 
variety $\VC=\Q_1\cap\cdots\cap\Q_5$ where $\Q_1,\ldots,\Q_5$ are quadrics of $\PG(5,q)$. 
\item In $\PG(5,q^2)$, the points of the variety-extension $\VC\star$ form the pointset of  a scroll with  
 ruling lines
$\{X(\conjB{X})^q\st X\in\Cplus\}.$
\end{enumerate}
\end{theorem}
  \begin{figure}[h]\caption{Lines of the scroll $\VC\star$}\label{fig2}
 \centering
 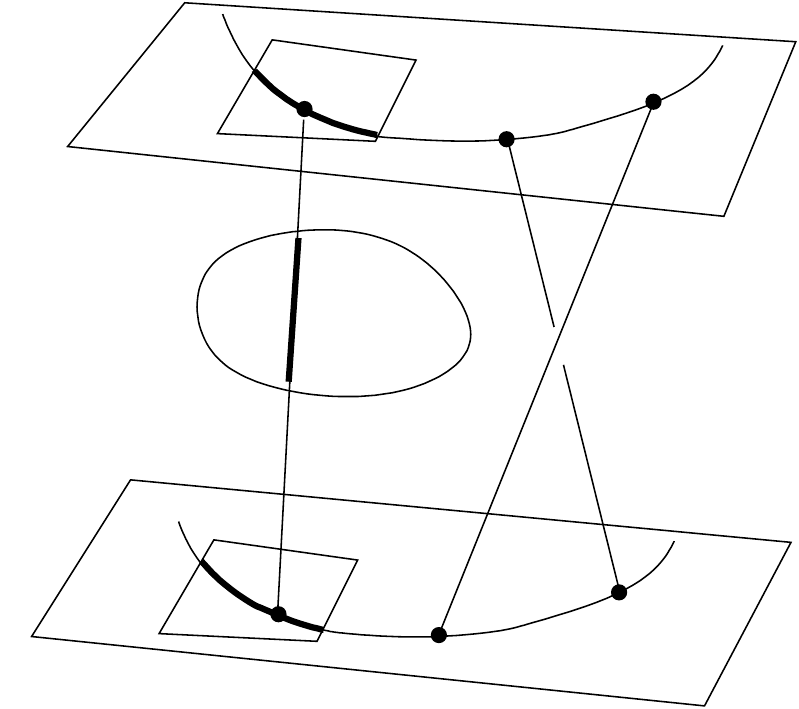
 \end{figure}

\begin{proof}
Let $\C$ be an $\Fq$-conic in the Baer subplane $\B$ of  the transversal plane $\Pit$ with $\Cplus$ denoting the unique $\Fqq$-conic  of $\Pit$ containing $\C$, that is $\C=\Cplus\cap\B$. 
By definition, in $\PG(5,q)$, $\Bo{\C}=\{XX^q\cap\PG(5,q)\st X\in\C\}$, $\Bo{\B}=\{XX^q\cap\PG(5,q)\st X\in\B\}$ and $\Bo{\Cplus}=\{XX^q\cap\PG(5,q)\st X\in\Cplus\}$, so 
\begin{eqnarray}\label{eqn:c1}
\Bo{\C}=\Bo{\Cplus}\cap \Bo{\B}.
\end{eqnarray}
By Result~\ref{result-HT}, the lines of  $\Bo{\B}$ form a scroll. As the lines of  $\Bo{\C}$ are a subset of the lines of $\Bo{\B}$, the  lines of $\Bo{\C}$ form  a scroll, ruled by the same homography as for the scroll $\Bo{\B}$. 

By Theorem~\ref{Fqq-conic-C}, the pointset of $\Bo{\Cplus}$ forms a variety $\VCplus=\Q_1\cap\Q_2$ where $\Q_i$ is a quadric   with homogenous equation $f_i=0$ of degree two over $\Fq$, $i=1,2$.
  By Result~\ref{result-HT},  the pointset of $\Bo{\B}$ forms a variety $\VB=\Q_3\cap\Q_4\cap\Q_5$ where $\Q_i$ is a quadric with homogenous equation $f_i=0$ of degree two over $\Fq$, $i=3,4,5$. So by (\ref{eqn:c1}), the pointset of $\Bo{\C}$ coincides with the pointset of  a variety $\VC$ which is the intersection of five quadrics, namely $
  \VC=(\Q_1\cap\Q_2)\cap(\Q_3\cap\Q_4\cap\Q_5).
$ Note that this is a variety of dimension 2 and order 4 by Lemma~\ref{conic-v42}. 
This completes the proof of part 1.

  The variety-extension of $\VC$ is $\VC\star=\newQonestar\cap\newQtwostar\cap\newQthreestar\cap\newQfourstar\cap\newQfivestar$, so in particular, 
  \begin{eqnarray}\label{eqn:yay}
  \VC\star=\VCplus\star\cap\VB\star.
  \end{eqnarray}

We now determine the points of $\VC\star$. 
By Theorem~\ref{Fqq-conic-C}, the points of  $\VCplus\star$ are the points of $\PG(5,q^2)$ on the lines 
\begin{eqnarray}\label{eqn:cplus}
\{XY^q\st X,Y\in\Cplus\}.
\end{eqnarray}
  By Corollary~\ref{pi0-Bose}, the points of  $\VB\star$ are the points  of $\PG(5,q^2)$  on the lines \begin{eqnarray}\label{eqn:B}
  \{X(\conjB{X})^q\st X\in\B\}.
  \end{eqnarray}
  By Lemma~\ref{lem:hypcong},  lines in (\ref{eqn:cplus}) and   (\ref{eqn:B}) either coincide, are disjoint, or meet in a point of $\Pit$ or $\Pit^q$. 
 Thus by (\ref{eqn:yay}),  $\VC\star$ consists of points on the set of lines which are in both (\ref{eqn:cplus}) and   (\ref{eqn:B}). That is,  $\VC\star$ consists of the points  of $\PG(5,q^2)$  on  the lines $\{X(\conjB{X})^q\st X\in\Cplus\}$.

 By Corollary~\ref{pi0-Bose}, the points of the variety-extension $\VB\star$ 
 form a scroll.
 As the lines of  $\VC\star$ are a subset of the lines of $\VB\star$, the  lines of $\VC\star$ form  a scroll, ruled by the same homography as for the scroll $\VB\star$. This completes the proof of part 2.
\end{proof}

Recall Remark~\ref{remark:BBexact} where we discuss the convention in the Bruck-Bose representation to not always include lines contained in $\si$ in our description, that is to not always be exact-at-infinity.  However, when looking at the Bruck-Bose representation  as a subset of the Bose representation, we 
 obtain the exact-at-infinity Bruck-Bose representation $\IBB=\IB\cap\Pi_g$.
In particular, the exact-at-infinity Bruck-Bose representation of a non-degenerate conic $\bar \C$ is the set 
$[\C]=\Bo{\C}\cap\Pi_g$. 
The pointset of $[\C]$ coincides with the pointset of the variety $\V([\C])=\VC\cap\Pi_g$.
We look at the extension of this variety, namely  $\V([\C])\star$.   
The next result   uses the exact-at-infinity setting, and so the variety   $\V([\C])\star$ may well contain lines in $\sistar=\langle g,g^q\rangle$, which   may account for the intersections with the transversal line $g$ of $\S$. 

\begin{corollary} \Label{cor-cplus-BB}
Let $\bar \C$ be an $\Fq$-conic in the Baer subplane $\bar \B$ of $\PG(2,q^2)$, then  $\bar P\in \bar\Cplus\cap\li $ if and only if in    the exact-at-infinity Bruck-Bose representation  in $\PG(4,q)$, 
 $P\in\V([\C])\star\cap g$.
\end{corollary}

\begin{proof}
We will show that $\Cplus\cap g=\V([\C])\star\cap g$.  Using the notation in the proof of Theorem~\ref{conic-subplane-A}, and intersecting both sides of (\ref{eqn:yay})  with  the transversal plane $\Pit$ gives
\begin{eqnarray}\label{eqn1}\VC\star\cap\Pit=\VCplus\star\cap\VB\star\cap\Pit.
\end{eqnarray}
By Corollary~\ref{pi0-Bose}, $\VB\star$ is ruled by  the lines $X\conjB{X}$ for $X\in\Pit$, and so  $\Pit\subset\VB\star$. Similarly, by Theorem~\ref{Fqq-conic-C}, $\Pit\cap\VCplus\star=\Cplus$. Hence
 (\ref{eqn1}) becomes $ \VC\star\cap\Pit=\Cplus$, and intersecting with $g$ yields 
\begin{eqnarray}\label{eqn2}
 \VC\star\cap g=\Cplus\cap g.\end{eqnarray}

Now let $\Pi_g$ be a 4-space whose extension meets $\Gamma$ in the line $g$. So we can construct the exact-at-infinity Bruck-Bose representation as 
 $\IBB=\IB\cap\Pi_g$. The Bruck-Bose representation of $\bar \C$ in the 4-space $\Pi_g$ is $[\C]=\Bo\C\cap\Pi_g$. Hence the pointset of $[\C]$ form a variety $\V([\C])=\VC\cap\Pi_g$.  In the quadratic extension, we have $\V([\C])\star=\VC\star\cap\Pigstar$, and intersecting with $g$ yields $\V([\C])\star\cap g=\VC\star\cap g$. Equating with (\ref{eqn2}) yields $\V([\C])\star\cap g=\Cplus\cap g$ as required.
\end{proof}

 \section{$\Fq$-conics in the  exact-at-infinity Bruck-Bose representation}

In the rest of this article, we use our results from the Bose representation setting to examine $\Fq$-conics in the Bruck-Bose representation. In particular, we give a geometric explanation as to why the notion of `special' arises in the Bruck-Bose representation.

 In \cite{BJW}, the authors show that an $\Fq$-conic $\bar \C$ contained in a Baer subplane tangent to $\li$ in $\PG(2,q^2)$ corresponds in the Bruck-Bose representation in $\PG(4,q)$ to a $g$-{\em special} normal rational curve of order 3 or 4. 
 By using the Bose representation, we show in Theorem~\ref{conic-quartic} that  in    the exact-at-infinity Bruck-Bose representation  in $\PG(4,q)$, the variety $\V([\C])$ is a quartic curve. This theorem also incorporates the case of an $\Fq$-conic in a secant Baer subplane. The full intersection of $\V([\C])$ with the hyperplane at infinity is straightforward to determine in this setting, and the details in the five distinct cases are given in  Corollary~\ref{cor:gspecial}.

\begin{theorem}\Label{conic-quartic} Let $\bar \C$ be an $\Fq$-conic in a Baer subplane $\bar \B$ of $\PG(2,q^2)$. 
Then in   the exact-at-infinity Bruck-Bose representation in $\PG(4,q)$, $\V([\C])$ is a quartic curve, which is either  a  non-degenerate conic and two lines; a twisted cubic and a spread line; or a 4-dimensional normal rational curve. 
\end{theorem}

\begin{proof} 
We work in the Bose representation of $\PG(2,q^2)$, so in $\PG(5,q^2)$, $\C$ is an $\Fq$-conic in the Baer subplane $\B$ of  the  transversal plane  $\Pit$.
As usual, let $X\mapsto \conjB{X}$ denote conjugacy in $\Pit$ with respect to  $\B$.  Let $g$ be a line of $\Pit$, and let $\Pi_g$ be a 4-space of $\PG(5,q)$ containing the 3-space $\si=\langle g,g^q\rangle\cap\PG(5,q)$. So $\Pi_g$ gives the Bruck-Bose representation of $\PG(2,q^2)$ with line at infinity  $\li$ corresponding to $g$. Moreover, $g,g^q$ are the transversal lines of the regular 1-spread $\S$ of $\si$.   

By Theorem~\ref{conic-subplane-A}, the points of $\lbose\C\rbose$ form a scroll, and lie on a variety $\VC$ which is the intersection of five quadrics, $\VC=\Q_1\cap\cdots\cap\Q_5$. 
As $\lbose\C\rbose$ is a scroll, by 
Lemma~\ref{conic-v42}, the pointset of $\lbose\C\rbose$ form a variety  $\VC=\V_2^4$ of $\PG(5,q)$. 
The 4-space $\Pi_g$ is a $\V^1_4$, and in $\PG(5,q)$, in general $\V_2^4\cap\V^1_4=\V^4_1$. So the Bruck-Bose representation of $\bar\C$ in the 4-space $\Pi_g$ is the set of points lying on the quartic curve $\V([\C])=\VC\cap\Pi_g=\Q_1\cap\cdots\cap\Q_5\cap\Pi_g$. 

 By Theorem~\ref{conic-subplane-A}, the variety-extension and scroll-extension of $\lbose\C\rbose$ coincide. That is, $\VC\star=\newQonestar\cap\cdots\cap\newQfivestar$  is the intersection of five quadrics; and the points of $\VC\star$ coincide with those on the scroll with ruling lines 
$\{X(\conjB{X})^q\st X\in\Cplus\}.$
Hence  the only lines the quartic curve $\V([\C])\star=\VC\star\cap \Pigstar$ 
 can contain are the scroll-lines, $X(\conjB{X})^q$ for $X\in\Cplus\cap g$. 

We now determine the precise form of the set $[\C]=\lbose\C\rbose\cap\Pi_g$. That is, in each of the possible cases, we determine the structure of the quartic curve $\V([\C])$. We consider the cases where $\bar \B$ is secant or tangent  to $\li$  separately. 
First suppose   $\bar \B$ is secant to $\li$, so in $\Pit$, $g$ is secant to $\B$. Using  Result~\ref{result-HT},  the set $[\B]=\lbose\B\rbose\cap\Pi_g$ consists of two parts. Firstly $[\B]$ contains   one of the ruling  planes of the Segre variety $\Bo{\B}$. Denote this plane by  $\alpha$,  and note that $\alpha$ meets $\si$ in a line that meets a regulus $\R$ of $\S$. The second part of $[\B]$ is the lines of the regulus $\R$. 
Hence the set 
 $[\C]=\lbose\C\rbose\cap\Pi_g$ meets $\alpha$ in a conic, denoted $
 \O$, and contains the lines (if any) of the regulus $\R$ that lie in $\Bo{\C}$. We  describe these lines and their relationship to $\O$  in more detail in the three cases where $g$ is secant, tangent and exterior to $\C$. 
If $g$ is a secant to $\C$ with $g\cap\C=\{P,Q\}$, then the quartic curve  $\V([\C])$ consists of  the two spread lines $\lbose P\rbose =[P]$, $\lbose Q\rbose=[Q]$, and the conic $\O$ which meets both $[P]$ and $[Q]$ in  a point. 
 If $g$ is a tangent to $\C$ with $g\cap\C=\{T\}$, then the quartic curve  $\V([\C])$ consists  of  the repeated spread line $\lbose T\rbose =[T]$, and the  conic $\O$ which is tangent to $\alpha\cap\si$ in a point of $[T]$.
If $g$ is  exterior to $\C$, then   $\Cplus\cap g=\{ X, {\conjB{X}}\}$ for some $X\in\Pit\setminus\B$.  In this case  
 the quartic curve  $\V([\C])$ consists of  the conic $\O$, and two lines which lie in the extension $\sistar\setminus\si$. 
By Theorem~\ref{conic-subplane-A} and Corollary~\ref{cor-cplus-BB}, these two lines are  the scroll-lines $X(\conjB{X})^q,\ \conjB{X} X^q$. 
 
Now suppose  $\bar \B$ is tangent to $\li$, so in $\PG(5,q^2)$, $g$ is tangent to $\B$. If $T=\B\cap g\in\C$, then  the quartic curve  $\V([\C])=\VC\cap\Pi_g$ contains the spread line $\lbose T\rbose=[T]$ and one point in each of the lines of $\lbose\C\rbose\setminus\lbose T\rbose$. Hence the quartic curve $\V([\C])$ is the spread line $\lbose T\rbose=[T]$ and a twisted cubic. 

If $T=\B\cap g\notin\C$, then the quartic curve  $\V([\C])=\VC\cap\Pi_g$ contains no spread line, and is disjoint from $\si=\langle g,g^q\rangle\cap\PG(5,q)$, hence is a 4-dimensional normal rational curve. Moreover, by Corollary~\ref{cor-cplus-BB}, this 4-dimensional normal rational curve   meets the line $g$ in the points $X,Y$, where $\Cplus$ meets $ g$ in the two points $\{ X, Y\}$ (possibly equal, possibly in an extension). 
\end{proof}

In particular, the proof describes the details of five sub-cases, which we summarise in the next corollary. 
\begin{corollary}
\Label{cor:gspecial}
  Let $\bar \C$ be an $\Fq$-conic in a Baer subplane $\bar \B$ of $\PG(2,q^2)$.  Then in exact-at-infinity Bruck-Bose representation  in $\PG(4,q)$, the pointset of $[\C]$ coincides with the pointset of a quartic curve denoted $\V([\C])$. 
\begin{enumerate}
\item Suppose  $\bar \B$ is secant to $\li$. 
\begin{enumerate}
\item If $\bar \C\cap \li=\{\bar P,\bar Q\}\in\bar\B$, then  $\V([\C])$ is  the spread lines $[P],[Q]$ together with a non-degenerate conic which meets $\si$ in a point of $[P]$ and a point of $[Q]$.
\item If $\bar \C\cap \li=\{\bar T\}\in\bar \B$, then $\V([\C])$ is a  non-degenerate conic which meets $\si$ in a repeated  point on $[T]$, together  with the repeated spread line $[T]$.
\item If $\bar \C\cap \li=\emptyset$, then $\bar \Cplus\cap \li=\emptyset$ and the intersections of $\bar \Cplus$ with $\li$ are two points $\{\bar X,\bar Y\}$ in $\PG(2,q^4)\setminus\PG(2,q^2)$, with  $\bar Y={{\bar X}^{{\bar{\mathsf c}}_{\scalebox{0.55}{\mbox{$\pi$}}}}}$. 
In this case, the curve  $\V([\C])$ is a  non-degenerate conic $\N_2$ disjoint from $\si$. Further, in $\PG(4,q^2)$, $\V([\C])\star$ contains the two  scroll-lines $X(\conjB{X})^q,\ \conjB{X} X^q$, and the two points $\Ntwostar\cap\sistar$ lie one on each scroll-line (and are not on $ g$ or $g^q$).
\end{enumerate}
\item Suppose $\bar \B$ is tangent to $\li$. 
\begin{enumerate}
\item If $\bar T=\bar\B\cap\li\in\bar  \C$, then  $\bar \Cplus\cap \li=\{\bar T,\bar L\}$ with $\bar T\neq \bar L$. The curve  $\V([\C])$ is 
a twisted cubic $\N_3$ together with the spread line $[T]$. Further, in $\PG(4,q^2)$, the three points of $\Nthreestar\cap\sistar$ consist of one real point on $[T]$, and the points $L\in g,L^q\in g^q$. 
\item  If   $\bar \C\cap \li=\emptyset$, then $\bar\Cplus$ meets $ \li$ in two points $\{\bar P,\bar Q\}$, possibly equal, possibly in the extension  $\PG(2,q^4)$. In this case,   $\V([\C])$  is a 4-dimensional normal rational curve $\N_4$. Further, in $\PG(4,q^2)$, the four points $\Nfourstar\cap\sistar$ are $P,Q\in g$, $P^q,Q^q\in g^q$   (possibly $P=Q$, possibly $P,Q$ lie in the extension $\PG(4,q^4)$).
\end{enumerate}
\end{enumerate}
\end{corollary}

\section{2-special normal rational curves in $\PG(4,q)$}\Label{sec-2special}

We now use the insight gained by working in the Bose representation of $\PG(2,q^2)$ in $\PG(5,q)$ to look at the Bruck-Bose representation of $\PG(2,q^2)$ in $\PG(4,q)$. Result~\ref{g-spec-know}(4) characterised $\Fq$-conics in tangent Baer subplanes
as $g$-special normal rational curves of order $r=3$ or 4 in $\PG(4,q)$. Here we refine the  notion of special to include the normal rational curves of order $r=2$, and so incorporate all five cases from Corollary~\ref{cor:gspecial}. That is, we give one characterisation which  includes all $\Fq$-conics in   both the secant and the tangent Baer subplane cases. This characterisation provides a better understanding of the structure of the curves at infinity. 

We define a {\em weight}, denoted $w(P)$, for each point  $P\in\sistar\cong\PG(3,q^2)$ in relation to the transversals of a regular spread $\S$. A point $P$ has  $w(P)=1$ if $P$ lies in one of the transversals of $\S$, otherwise $w(P)=2$. 
Similarly  a point $P$ in $\sistarstar\cong\PG(3,q^4)$ has  $w(P)=1$ if $P$ lies in $\gstar$ or $\gqstar$, otherwise $w(P)=2$.

\begin{defn}
 Let $\N_r$ be a normal rational curve of order $r$ in $\PG(4,q)$, so $\N_r$ meets 
 $\si$, the hyperplane at infinity, in $r$ points  
$\{P_1,\ldots,P_r\}$, these points may be repeated, or in an extension. We say $\N_r$ is a {\em 2-special normal rational curve} if $w(P_1)+\cdots+w(P_r)=4$. 
 \end{defn}
 

This 2-special definition leads to a single short characterisation of all $\Fq$-conics in $\PG(2,q^2)$. 
Note that while  Corollary~\ref{cor:gspecial} gives the exact-at-infinity correspondence, in the next theorem we use the
 usual Bruck-Bose convention where we do not include lines at infinity   in our description (see Remark~\ref{remark:BBexact}). That is, for an $\Fq$-conic $\bar \C$ of $\PG(2,q^2)$, in our description of the Bruck-Bose representation of $[\C]$, we only include the non-linear component of the quartic curve $\V([\C])$. Hence by Theorem,~\ref{conic-quartic}, $[\C]$ is a non-degenerate conic, a twisted cubic or a 4-dimensional normal rational curve, that is, a $k$-dimensional normal rational curve with $k$ equal to 2, 3 or 4. Note we do not need the variety notation in this description as a normal rational curve is a variety, so there is no confusion about which variety we are extending. 

\begin{theorem}\Label{thm:Fqconic-unify}
A set $\bar\C$ in $\PG(2,q^2)$ is an $\Fq$-conic if and only if the corresponding set $[\C]$ in the Bruck-Bose $\PG(4,q)$ representation is 
 a 2-special normal rational curve.
\end{theorem}

\begin{proof} First suppose that $\bar\C$ is an $\Fq$-conic in a Baer subplane $\bar\B$ in $\PG(2,q^2)$, there are five cases to consider for $\bar\C$,  these are listed in  Corollary~\ref{cor:gspecial}, and we use the same numbering and notation here. As noted above,  our Bruck-Bose description of $[\C]$ only includes the non-linear component of the quartic curve $\V([\C])$. That is, 
 we omit any lines (which are contained in $\si$) in our description of $[\C]$; so by   Corollary~\ref{cor:gspecial}, $[\C]$ is a 
 $k$-dimensional normal rational curve, $k=2$, 3 or 4.

Case 1, suppose $\bar \B$ is secant to $\li$, then in $\IBB$, $[\B]$ is a plane not containing a spread line, and $[\C]$ is a conic in $[\B]$. We show that $[\C]$ is 2-special. In Case 1(a), $\bar \C\cap\li=\{\bar P, \bar Q\}\subset \bar \B$, so $[\C]$ meets $\si$ in two points $X=[P]\cap[\B]$, $Y=[Q]\cap[\B]$. The points $X,Y$ both have weight two,   so the points of $[\C]$ in $\si$ have weights summing to four. That is, $[\C]$ is a 2-special normal rational curve of order 2.  
In Case 1(b), $\bar \C\cap\li=\{\bar T\}\in\bar \B$, and  in $\IBB$, $[\C]\cap\si$ is the repeated point $Z=[T]\cap[\B]$. The point $Z$ has weight 2, which we count twice as $Z$ is a repeated point, so the weights sum to four and $[\C]$ is 2-special. 
In case  1(c),  by Corollary~\ref{cor:gspecial},   in $\IBB$, $[\C]$ is a conic whose extension  meets $\sistar$ in two points of weight 2, and so $[\C]$ is 2-special. 
For Case 2, suppose $\bar \B$ is tangent to $\li$.
In case 2(a), by Corollary~\ref{cor:gspecial},  in $\IBB$, $[\C]$ is a twisted cubic that meets $\si$ in one real point which has weight 2, and whose extension meets $\sistar$ in a point on $g$ and the conjugate point on $g^q$, both of weight 1. The weights  of these three points sum to four, hence  $[\C]$ is 2-special. In case 2(b),  by Corollary~\ref{cor:gspecial},  in $\IBB$, $[\C]$ is a 4-dimensional normal rational curve, whose extension meets $\si$ in four points which  lie on the transversals $g,g^q$ (or the extensions $\gstar,\gqstar$). Hence their weights sum to four, and $[\C]$ is 2-special. 

Conversely, let $\N_r$ be a 2-special normal rational curve of order $r$ in $\PG(4,q)$. We consider the four cases $r=1,2,3,4$ separately.
Note that as $\N_r$ is a curve defined over $\Fq$, if a point $P\in\sistar\setminus\si$ lies in the extension of $\N_r$, then so does the conjugate point $P^q$. 

Suppose $r=4$, so $\N_4$ is a 4-dimensional normal rational curve which meets $\si$ in four points $P_1,P_2,P_3,P_4$, which may be in the extension $\sistar$ or $\sistarstar$. As $\N_4$ is 2-special, $w(P_1)+w(P_2)+w(P_3)+w(P_4)=4$, and so  $w(P_i)=1$, $i=1,2,3,4$. So the four points lie on  the transversals $g,g^q$ (or the extensions $\gstar,\gqstar$). 
Hence $\N_4$ meets $g$ (or the extension $\gstar$) in two points $P,Q$ (possibly repeated) and meets $g^q$ (or $\gqstar$) in $P^q,Q^q$. 
Note that if $\N_4^\blackstar$ meets $\gstar\setminus g$ in two points, they have form $P,P^{q^2}$, and $\N_4^\blackstar$ meets $\gqstar\setminus g^q$ in the two points  $P^q,P^{q^3}$. 
By \cite[Thm 5.9]{BJW}, $\N_4$ corresponds in $\PG(2,q^2)$ to an $\Fq$-conic $\bar\C$ in a tangent Baer subplane $\bar\B$ with $\bar\B\cap\li\notin\bar\C$.

Suppose $r=3$, so $\N_3$ is a 3-dimensional normal rational curve which meets $\si$ in three points $P_1,P_2,P_3$ which may be in an extension.
As $\N_3$ is 2-special, 
 $w(P_1)+w(P_2)+w(P_3)=4$, and so $w(P_1)=2$, say, and $w(P_2)=w(P_3)=1$.   That is, $\N_3$ meets $\si$ in one real point $P_1\in\PG(3,q)$, one point  $P_2\in g$ and the conjugate point  $P_3=P_2^q\in g^q$. 
By  \cite[Thm 5.6]{BJW},  $\N_3$ corresponds  in $\PG(2,q^2)$  to an $\Fq$-conic $\bar\C$ in a tangent Baer subplane $\bar\B$ with $\bar\B\cap\li\in\bar\C$.

Suppose $r=2$, so $\N_2$ is a non-degenerate conic  which lies in a plane $\alpha$ not containing a spread line, and $\N_2$  meets $\si$ in two points $P,Q$ 
with $w(P)+w(Q)=4$, so $w(P)=w(Q)=2$. The plane $\alpha$ corresponds in $\PG(2,q^2)$ to a Baer subplane, denoted $\bar \B$, which is secant to $\li$, and $\N_2$ corresponds to an $\Fq$-conic denoted $\bar\C$ contained in $\bar\B$.
For completeness in the $r=2$ case, we look at the  three possibilities for the points $P,Q$ in more detail. The three cases are:
(a) $P,Q\in\si$ are distinct; (b) $P=Q\in\si$; (c) $P,Q\in\sistar\setminus\si$. 
In case (a), the points $P,Q$ lie on distinct spread lines which we denote $[X], [Y]$. In $\PG(2,q^2)$, the $\Fq$-conic $\bar \C$
 contains the corresponding points $\bar X,\bar Y\in\bar \B \cap \li$. 
In case (b), the point $P$ lies on a spread line which we denote $[X]$. In $\PG(2,q^2)$, the $\Fq$-conic $\bar \C$
 is tangent to $\li$ at the corresponding point $\bar X\in\bar\B$. 
In case (c), as $P,Q\in\sistar\setminus\si$, we have $Q=P^q$. Let $m=\alpha\cap\si$. 
 As the plane $\alpha$ corresponds to a Baer subplane of $\PG(2,q^2)$, $m$ meets $\S$ in the lines of a regulus which we denote $\R$. Note also that 
 $P,P^q\in \mstar$. The $q+1$ lines of $\R$ extended to $\sistar$ meet $g$ in $q+1$ points  that form 
a Baer subline  denoted $\bs$. 
 By 
 Result~\ref{subline-star}, the extension of the regulus $\R$ is $\R\star=\{X(\conjb{X})^q\st X\in g\}$. As  $\mstar$ meets each line of $\R\star$,
  $P$ lies on a line of form $X(\conjb{X})^q$, for some $X\in g$, and   $P^q\in X^q\conjb{X}$. The Baer subline $\bs$ of $g$ corresponds in $\PG(2,q^2)$ to a Baer subline $\bar \bs$ of $\li$. Further, $\bar \bs=\bar \B\cap\li$,  so ${ {\bar{\mathsf c}_{b}}}$ and 
  ${{\bar{\mathsf c}}_{\scalebox{0.55}{\mbox{$\pi$}}}}$ 
  coincide for points on $\li$, that is, ${{\bar X}^{ {\bar{\mathsf c}_{b}}}}={{\bar X}^{{\bar{\mathsf c}}_{\scalebox{0.55}{\mbox{$\pi$}}}}}$.  
  By \cite[Thm 5.4]{BJW},  $\N_2$ corresponds in $\PG(2,q^2)$ to an $\Fq$-conic $\bar\C$ in a secant Baer subplane $\bar\B$ 
 and $\bar\Cplus$ meets $\li$ in the two points $\bar X,{{\bar X}^{{\bar{\mathsf c}}_{\scalebox{0.55}{\mbox{$\pi$}}}}}$.  
 
 If $r=1$, then $\N_r$ is a line which either lies in $\si$, or meets $\si$ in one point. In either case, the weights do not sum to four, so we do not have a 2-special normal rational curve. 
 \end{proof}

%
%
%
%

\section{Conclusion}

We set out to gain a geometric understanding of why characterisations of the  Bruck-Bose representation of objects in $\PG(2,q^2)$ relate to the transversal lines of the associated regular spread. By looking at the Bruck-Bose representation as a subset of the Bose representation, we have given geometric arguments explaining this notion of specialness.

 Further, we have used geometric techniques to characterise normal rational curves of $\PG(4,q)$ that correspond to $\Fq$-conics of $\PG(2,q^2)$.  

While this article focussed on $\Fq$-conics of $\PG(2,q^2)$, these arguments can be extended to other $\Fq$-varieties of $\PG(2,q^2)$. 
We define an \emph{$\Fq$-variety $\bar \V$ of $\PG(2,q^2)$} to be a variety which  is projectively equivalent to a variety in $\PG(2,q)$. An $\Fq$-variety $\bar \V$ is contained in a unique $\Fqq$-variety, denoted $\bar {\V^\plus}$, which consists of the points of $\PG(2,q^2)$ satisfying the same homogeneous equations that define $\bar \V$.
 The techniques from Section~\ref{sec:Baer} can be generalised to show that $\Fq$-varieties yield  $g$-special sets in  the Bruck-Bose representation in $\PG(4,q)$ in the following sense. 
If $\bar {\V^\plus}\cap\li=\{\bar P_1,\ldots,\bar P_k\}$, then  in    the exact-at-infinity Bruck-Bose representation  in $\PG(4,q)$, the pointset of $[\V]$ forms a variety whose extension  meets the transversal  $g$ of the regular spread $\S$ in the points 
$\{ P_1,\ldots,P_k\}$.

\bigskip\bigskip

{\bfseries Author information}

S.G. Barwick. School of Mathematical Sciences, University of Adelaide, Adelaide, 5005, Australia.
susan.barwick@adelaide.edu.au

W.-A. Jackson. School of Mathematical Sciences, University of Adelaide, Adelaide, 5005, Australia.
wen.jackson@adelaide.edu.au

P. Wild. Royal Holloway, University of London, TW20 0EX, UK. peterrwild@gmail.com

\end{document}